\documentclass[12pt]{article}
\usepackage{amsmath,amssymb,amsthm, amsfonts}
\usepackage{etaremune, enumerate, float, verbatim}
\usepackage{hyperref}
\usepackage{graphicx}
\usepackage{url}
\usepackage[mathlines]{lineno}
\usepackage{dsfont} 
\usepackage{graphicx, subcaption, caption}
\usepackage[algo2e,ruled,vlined]{algorithm2e}
\usepackage{color}
\definecolor{red}{rgb}{1,0,0}

\definecolor{blue}{rgb}{0,.6,.9}

\definecolor{green}{rgb}{0,.6,0}

\definecolor{purp}{rgb}{.5,0,.5}

\numberwithin{figure}{section}   

\newtheorem{thm}{Theorem}[section]
\newtheorem{cor}[thm]{Corollary}
\newtheorem{lem}[thm]{Lemma}
\newtheorem{prop}[thm]{Proposition}
\newtheorem{conj}[thm]{Conjecture}

\theoremstyle{definition}
\newtheorem{rem}[thm]{Remark}

\theoremstyle{definition}

\theoremstyle{definition}
\newtheorem{ex}[thm]{Example}

\newcommand{\rad}{\operatorname{rad}} 

\newcommand{\thc}{\operatorname{th}_c}
\newcommand{\thcx}{\operatorname{th}_c^{\x}}

\newcommand{\capt}{\operatorname{capt}} 

\newcommand{\dist}{d}
  
\newcommand{\diam}{\operatorname{diam}}

\newcommand{\x}{\times}

\newcommand{\bit}{\begin{itemize}}
\newcommand{\eit}{\end{itemize}}
\newcommand{\ben}{\begin{enumerate}}
\newcommand{\een}{\end{enumerate}}
\newcommand{\beq}{\begin{equation}}
\newcommand{\eeq}{\end{equation}}
\newcommand{\bea}{\begin{eqnarray*}} 
\newcommand{\eea}{\end{eqnarray*}}
\newcommand{\bpf}{\begin{proof}}
\newcommand{\epf}{\end{proof}\ms}
\newcommand{\bmt}{\begin{bmatrix}}
\newcommand{\emt}{\end{bmatrix}}
\newcommand{\ms}{\medskip}

\newcommand{\lc}{\left\lceil}
\newcommand{\rc}{\right\rceil}
\newcommand{\lf}{\left\lfloor}
\newcommand{\rf}{\right\rfloor}
\newcommand{\du}{\,\dot{\cup}\,}
\newcommand{\noi}{\noindent}
\newcommand{\lp}{\!\left(}
\newcommand{\rp}{\right)}

\title{
Optimizing the trade-off between number of cops and capture time in Cops and Robbers
}
\author{Anthony Bonato\thanks{Ryerson University, Toronto, ON, Canada. (abonato, sean.english)@ryerson.ca.}\and
Jane Breen\thanks{Department of Mathematics, Iowa State University, Ames, IA 50011, USA, jane.breen@uoit.ca, (geneson, jmsdg7, hogben, reinh196)@iastate.edu.}\and Boris Brimkov\thanks{Department of Mathematics and Statistics, Slippery Rock University, Slippery Rock, PA 16057, USA,  boris.brimkov@sru.edu.}\and Joshua Carlson\footnotemark[2]\and Sean English\footnotemark[1]\and Jesse Geneson\footnotemark[2]\and Leslie Hogben\footnotemark[2]\ \thanks{American Institute of Mathematics, 600 E. Brokaw Road, San Jose, CA 95112, USA, hogben@aimath.org} 
\and K.E. Perry\thanks{Department of Mathematics, University of Denver, Denver, CO 80208, USA,  kperry@soka.edu.}\and Carolyn Reinhart\footnotemark[2]}

\begin{document}
\maketitle

\vspace{-10pt}

\begin{abstract} 
The cop throttling number $\thc(G)$ of a graph $G$ for the game of Cops and Robbers is  the minimum of $k + \capt_k(G)$, where $k$ is the number of cops and $\capt_k(G)$ is the minimum number of rounds needed for $k$ cops to capture the robber on $G$ over all possible games in which both players play optimally. 
In this paper, we construct a family of graphs having $\thc(G)= \Omega\lp n^{2/3}\rp$, establish a sublinear upper bound on the cop throttling number, and show that the cop throttling number of chordal graphs is $O\lp\sqrt{n}\rp$. We also introduce the product cop throttling number $\thcx(G)$ as a parameter that minimizes the person-hours used by the cops. 
This parameter extends the notion of speed-up that has been studied in the context of parallel processing and network decontamination.
We establish bounds on the product cop throttling number in terms of the cop throttling number, characterize graphs with low product cop throttling number, and show that for a chordal graph $G$, $\thcx(G)=1+\rad(G)$.
\end{abstract}

\noi {\bf Keywords} Cops and Robbers, throttling, product throttling,  chordal graph, graph searching

\noi{\bf AMS subject classification} 05C57, 91A43 

\section{Introduction}\label{sintro}


The game of Cops and Robbers is a perfect information two-player game played on a graph $G$ on $n$ vertices.  One player controls a team of cops and the other controls a single robber. The game starts with the cops choosing a multiset of vertices to occupy, and then the robber chooses a vertex to occupy.  In each round of the game, first each cop moves to a neighbor of the vertex they currently occupy or remains at the same vertex, and then the robber moves analogously. The aim for the cops is to \emph{capture} the robber (that is, move to the same vertex that the robber currently occupies), and the aim for the robber is to evade capture. 
The game with a single cop was first introduced independently in \cite{NW83, Q78}. Graphs for which a single cop always has a winning strategy are called \emph{cop-win}. This was extended to the idea of having more than one cop, and the \emph{cop number} $c(G)$ of a graph $G$ is defined as the minimum number of cops required to capture the robber on $G$ \cite{AF84}. Meyniel's conjecture states that for any graph on $n$ vertices, $c(G) = O(\sqrt{n})$  \cite{frankl1987cops}.  For more background on Cops and Robbers, the reader is directed to \cite{CRbook, BP}.

Other questions may be asked of the game of Cops and Robbers, such as the \emph{capture time}, denoted $\capt(G)$, which is the number of rounds it takes  for $c(G)$ cops to capture the robber on the graph $G$, assuming all players follow optimal strategies \cite{BGHK09}. Capture time was generalized further in \cite{BPPR17} to consider the case where more cops  than necessary are used. That is, for any $k \geq c(G)$, the \emph{$k$-capture time} of $G$, denoted $\capt_k(G),$ is the minimum number of rounds it takes for $k$ cops to capture the robber on $G$, assuming that all players follow optimal strategies. It is interesting to explore the tradeoff between  the number of  cops and the capture time, which led to  the introduction of \emph{throttling} for the game of Cops and Robbers in \cite{CRthrottle}.

As in \cite{CRthrottle}, the \emph{cop throttling number} of a graph $G$ is denoted $\thc(G)$, and is defined as
\[ \thc(G) = \min_k\{ k + \capt_k(G)\},\]
where it is assumed that if $k<c(G)$, then the $k$-capture time is infinite.   
It is known  that $\thc(G) = O(\sqrt{n})$ for several families of graphs $G$; in particular, this was shown for trees, unicyclic graphs, some Meyniel extremal families, and several others in \cite{CRthrottle}. It was also asked in that paper whether $\thc(G)= O(\sqrt n)$ for all graphs.  We answer this question in the negative by exhibiting a family of graphs $H_n$ of order $n$ with $\thc(H_n)= \Omega\lp n^{2/3}\rp$, and we establish a sublinear upper bound for the cop throttling number  (see Section \ref{Lthrottle}).  In Section~\ref{s:chord}, we  prove that for any chordal graph of order $n$ the $k$-capture time is equal to the $k$-radius and the cop throttling number  is $O\lp\sqrt{n}\rp$. We also answer an open problem from \cite{topdir} about classifying cop-win outerplanar graphs.

The cop throttling number, which optimizes the sum of the resources used to accomplish a task and the time to accomplish the task, follows in the established study of throttling for other parameters (cf. \cite{powerdom-throttle, BY13throttling, Carlson, PSDthrottle}).  In the case of Cops and Robbers, arguably it is the person-hours that should be optimized, i.e., the product rather than the sum.   In Section~\ref{sprodthrot}, we study the problem of  optimizing the product of the resources used to accomplish a task and the time needed to complete that task. Note that if one minimizes the product $k\capt_k(G)$ over $k$, the minimum is always 0, achieved by $k=n$, where $n$ is the order of the graph $G$. Not only is this trivial, it is also misleading from a practical perspective because there is certainly a real cost to placing a cop on a vertex.   Thus, 
we define the {\em product cop throttling number}  of a graph $G$ by 
\[ \thcx(G) =\min_k
\{k\left(1+\capt_k(G)\right)\!\}.
\] 
We follow the literature in using {\em cop throttling} to refer to throttling the sum, whereas when throttling the product, the word {\em product} is always explicitly included. 
The notion of product throttling has implicitly been studied in \cite{luccio-pagli-conference,luccio-pagli}, where the authors investigate the speed-up obtained by using a larger number of cops  when chasing the robber on grids and tori. In particular, they  define \emph{work} as $w_k=k \cdot \capt_k(G)$, and the \emph{speed-up} between using $j$ and $i$ cops, $j>i$, as $w_i/w_j$. They show that a super-linear speed-up may occur in certain classes of graphs. Our study of product cop throttling extends this idea by considering the number of cops that yields the largest possible speed-up. A similar study \cite{luccio-pagli-decontamination} in the context of network  decontamination shows that larger teams of agents may decrease the  overall work done to decontaminate a network. More generally, the  concepts of work, speed-up, and related optimization problems are  common in the design and analysis of parallel algorithms (see, e.g., \cite{karp-ramachandran} and the bibliography therein).

 In Section \ref{sprodthrot}, we establish bounds on the product cop throttling number in terms of the cop throttling number, characterize graphs with low product cop throttling number, show that $\thcx(G) =1+\rad(G)$  for any chordal graph $G$ (implying  that the product throttling number can be linear in the order of the graph), and 
construct a family of graphs $M(\ell)$,  where $\thcx(M(\ell))$ cannot be realized by any set of cardinality $c(M(\ell))$ nor by any set of cardinality $\gamma(M(\ell))$.       


Throughout, we assume $G$ is a  simple undirected graph on $n$ vertices. Definitions of standard graph theory terms can be found in \cite{GTbook}.
 We refer to a multiset $S$ of vertices of $G$ as a \emph{capture set} if $|S|\ge c(G)$, since placing the cops on the vertices of $S$ ensures that the robber will be captured in a finite number of rounds. 
 As in \cite{CRthrottle}, $\capt(G;S)$ is defined to be the maximum number of rounds until the robber is captured (over all possible robber placements) with the cops starting on the vertices of $S$,   $\thc(G;S) = |S| + \capt(G;S)$, and $\thcx(G;S) = |S|(1+\capt(G;S))$.   With this notation, $\thc(G) = \min_{S \subseteq V(G)} \thc(G;S)$ and $\thcx(G) = \min_{S \subseteq V(G)} \thcx(G;S)$; note that the notation $A\subseteq B$ is applied to multisets.
For $k\ge c(G)$, it is also convenient to define $\thc(G,k)=\min_{|S|=k} \thc(G;S)$ and $\thcx(G,k)=\min_{|S|=k} \thcx(G;S)$.  With this notation, $\thc(G) = \min_{k} \thc(G,k)$ and $\thcx(G) = \min_{k} \thcx(G,k)$. Recall that the $k$-radius of a graph G is defined  to be \vspace{-3pt}
 \[\rad_k(G)=\min_{S\subseteq V, |S|=k}\max_{v\in V}\dist(v,S).\vspace{-3pt}\]

An induced subgraph $H$ of $G$ is a \emph{retract} of a graph $G$ if there is a mapping $\varphi:V(G)\to V(H)$ whose restriction to $V(H)$ is the identity and such that $uv\in E(G)$ implies $\varphi(u)\varphi(v)\in E(H)$ or $\varphi(u)=\varphi(v)$; such a mapping $\varphi$ is called a {\em retraction}.  
	  The robber's {\em shadow} on a retract is the image of the robber under the retraction.

 \section{Bounds for cop throttling number}\label{Lthrottle}
We  begin this section by answering negatively the question of whether  $\thc(G) =O(\sqrt{n})$ for all graphs $G$ with $n$ vertices \cite[Question 4.5]{CRthrottle}.   We then establish  the first sublinear upper bound on the cop throttling number for all connected graphs.  Finally, we improve the upper bound  given in \cite{CRthrottle} for cop throttling number for unicyclic graphs.


\subsection{Graphs with high cop throttling number}\label{sshighsumthrot}

 In this section,  we construct a family of graphs of order $n$ with throttling number  $\Omega(n^{2/3})$. We first prove a more general result that  implies the $\Omega(n^{2/3})$ bound  by using the existence of graphs with cop number $\Omega\lp\sqrt n\rp$.  This  result could also be used to improve this lower bound on the maximum cop throttling number if in the future,  Meyniel's conjecture is disproved.

\begin{thm} \label{mainhighth}
 Suppose that there exists a family of connected graphs of all orders $n$ with cop number $\Omega(n^{\alpha})$ for    a fixed real number $\alpha\in[\frac{1}{2} , 1)$. Then   there exist connected graphs $H_n$ on $n$ vertices with   $\thc(H_n)=\Omega\lp n^{1/(2-\alpha)}\rp\!$.
\end{thm}

\begin{figure}[h]
\centering
\scalebox{.4}{\includegraphics{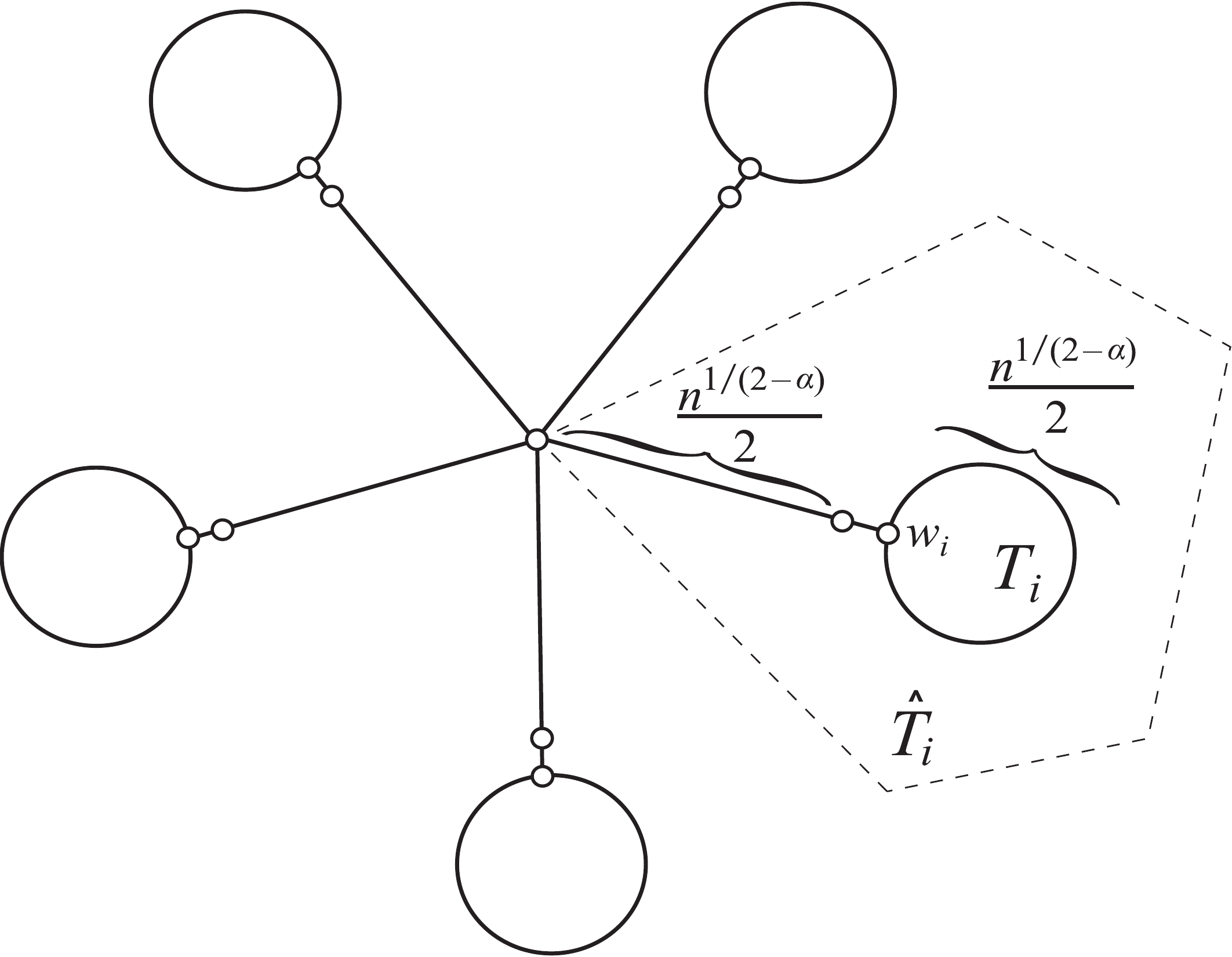}}
\caption{ Construction of the  connected graph $H_n$ of order $n$ with cop throttling number $\Omega\lp n^{1/(2-\alpha)}\rp$.  \label{highth}}
\end{figure}

\begin{proof}
By assumption, there exists a constant $b$ such that
 there exists a connected graph $Q(n)$ on $n$ vertices with  $c(Q(n))\ge b n^{\alpha}$. We assume $n$ is sufficiently large that the distinction between floor and ceiling does not matter except where marked, and thus treat quantities as integers.

A {\em spider} is a tree with exactly one vertex of degree three or more, called the {\em body vertex}. 
Start with a spider  on $n$ vertices in which there are approximately $n^{(1-\alpha)/(2-\alpha)}$ legs each of length approximately $n^{1/(2-\alpha)}$ and let $v$ be the body vertex.  
Form a new graph $H_n$  by  replacing in each leg
 the $\frac 1 2 n^{1/(2-\alpha)}$ vertices farthest from  $v$ 
 by a subgraph $T_{i}=Q(\frac 1 2 n^{1/(2-\alpha)})$. The subgraph $T_i$ is connected by one edge from some vertex $w_{i}$ of $T_i$  to the end of the leg that remains.  
Figure \ref{highth} depicts the construction of $H_n$.  The subgraph  $\hat T_i$ is the component of $H_n-v$ that contains $T_i$.  Observe that  $c(T_i)\ge b\lp\frac 1 2 n^{1/(2-\alpha)}\rp^{\alpha}=\frac b {2^\alpha} n^{\alpha/(2-\alpha)}$.

We show that $\thc(H_n)=\Omega\lp n^{1/(2-\alpha)}\rp$.   If   $S$ is a set of cops with   $|S|\ge \frac b {2^\alpha} n^{1/(2-\alpha)}$, then $\thc(H_n;S)=\Omega(n^{1/(2-\alpha)})$.  So  assume that   $|S|< \frac b {2^\alpha} n^{1/(2-\alpha)}$.  By the pigeonhole principle, there exists a subgraph $\hat T_{j}$ that initially has at most 
\[\lc\frac{\frac b {2^\alpha} n^{1/(2-\alpha)}}{n^{(1-\alpha)/(2-\alpha)}}\rc-1=\lc\frac b {2^\alpha} n^{\alpha/(2-\alpha)}\rc-1<\frac b {2^\alpha} n^{\alpha/(2-\alpha)}\] 
cops.   If the robber starts on $T_{j}$, then the robber can evade capture as long as there are at most  $\frac b {2^\alpha} n^{\alpha/(2-\alpha)}-1< c(T_j)$ cops on $T_{j}$. The robber can just use the same strategy that they would to avoid capture by  $\frac b {2^\alpha} n^{\alpha/(2-\alpha)}-1$ cops in $T_j$. Thus, the robber is safe at least until some cop who was initially outside $\hat T_{j}$ reaches $w_{j}$. In this case, the capture time is at least $\frac 1 2 n^{1/(2-\alpha)}$, which gives  $\thc(H_n)=\Omega\lp n^{1/(2-\alpha)}\rp$.   \end{proof}

The preceding theorem has many interesting applications. Since it is known that there exist graphs with  cop number $\Omega(n^{1/2})$, we can apply Theorem \ref{mainhighth} to 
produce a negative answer to the second part of Question 4.5 in \cite{CRthrottle}, which asked whether $\thc(G)=O(\sqrt n)$ for all graphs $G$ of order $n$.

\begin{thm}\label{t:Mext}{\rm\cite[Theorem 3.8]{CRbook}, \cite{praalat2010does}} There exist connected graphs on $n$ vertices with cop number at least $\sqrt{\frac n 8}$ for all $n \geq 72$.
\end{thm}

The next result follows immediately by applying Theorem \ref{mainhighth} with $\alpha= 1/2$. 

\begin{cor}\label{omega23}
There exist connected graphs of all orders $n$ with cop throttling number  $\Omega(n^{2/3})$.
\end{cor}

The proof of Theorem \ref{mainhighth} actually works for a variable value of $\alpha$ tending to one.
	More specifically, let $\alpha(x)$ be a continuous eventually non-decreasing function with $\frac{1}{2} \leq \alpha(x) \leq 1-\frac{\log\log x}{\log x}$ such that there exist connected graphs of order $n$ with cop number $\Omega(n^{\alpha(n)})$. Then there exist connected graphs on $n$ vertices with cop throttling number at least $\Omega\lp \ell(n)\rp$ where $\ell(n)\in [1,n]$ is defined to be the solution to the equation
	\[
	x=\frac 12 n^{1/(2-\alpha(x))}.
	\]

It is interesting to consider the ratio of maximum throttling number to maximum cop number.  We define $mc(n)$ and $mt(n)$, respectively, as the maximum cop number and maximum cop throttling number over all connected graphs of order $n$.

\begin{conj}\label{cj:mt/mc}
$\lim_{n \rightarrow \infty} \frac{mt(n)}{mc(n)} = \infty$.
\end{conj}

Conjecture \ref{cj:mt/mc} would follow from Theorem \ref{mainhighth} if it is proven that the maximum possible cop number of a connected graph is  $\Theta(n^{\alpha})$ for some fixed $\alpha < 1$ (Meyniel extremal families imply $\alpha \ge \frac 1 2$):  Suppose $mc(G)=\Theta(n^{\alpha})$ for graphs of order $n$, which implies there is a family of graphs $G_n$ of order $n$ such that $c(G_n)=\Theta(n^\alpha)$.   Then, by Theorem \ref{mainhighth},  
$\thc(G_n)=\Omega\lp n^{1/(2-\alpha)}\rp$, and $\frac 1{2-\alpha}>\alpha$ for  $\frac 1 2\le \alpha  <1$. 

The final application of Theorem \ref{mainhighth} we will mention is that it opens a new way to attack Meyniel's conjecture. Indeed, if it can be shown that cop throttling numbers are $O(n^{2/3})$, this would suffice to show that cop numbers are $O(\sqrt{n})$.


\subsection{Sublinear upper bound on the cop throttling number}

We begin with some definitions and lemmas used to establish a sublinear bound on the cop throttling number.
Given a connected graph $G$, a $u$-$v$ \emph{geodesic} is a shortest path between vertices $u$ and $v$. A {\em geodesic} is a path that is a $u$-$v$ geodesic for some choice of $u$ and $v$. Observe that a geodesic  is a retract and an induced subgraph. 

An induced subgraph $H$ of $G$ is \emph{$k$-guardable} if after finitely many moves, $k$ cops can arrange themselves in $H$ so that the robber is immediately captured upon entering $H$. For example, a clique is $1$-guardable. After some round, we  say a $k$-guardable subgraph $H$ is \emph{guarded} if for the rest of the game, some set of cops in $H$ stay in position to immediately capture the robber upon entering $H$.


\begin{lem}\label{lemma guard path exact}
If $P$ is a geodesic of length $k$, then for any $r\geq 1$, we can place $\left\lceil\frac{k+1}{2r+1}\right\rceil$ cops on $P$ such that $P$ will be guarded in at most $r$ steps. Further, after these $r$ steps, only one cop is necessary to continue guarding $P$.
\end{lem}

\begin{proof}
It suffices for the cops to capture the robber's shadow on $P$  and for the cop that captures the shadow to stay on it. 
By \cite{BPPR17}, $\capt_{\left\lceil\frac{k+1}{2r+1}\right\rceil}(P)=\rad_{\left\lceil\frac{k+1}{2r+1}\right\rceil}(P) = r$, so the robber's shadow is caught in at most $r$ steps. Note that if $P=(v_1,v_2,\dots,v_{k+1})$, we can place one cop at $v_{r+1+(2r+1)j}$ for each $0\leq j\leq \left\lceil\frac{k+1}{2r+1}\right\rceil-1$, so that every vertex on $P$ is within distance $r$ from some cop.
\end{proof}

It is straightforward to see that Lemma~\ref{lemma guard path exact} is sharp since if only $\left\lceil\frac{k+1}{2r+1}\right\rceil-1$ cops are placed on a path with $k+1$ vertices, there will be a vertex at distance at least $r+1$ from every cop. This level of precision has a negligible effect on the proof of Theorem~\ref{main}, however, so we state an immediate corollary of this result that is weaker but easier to use.

\begin{lem}\label{lemma guard path}
If $P$ is a geodesic of length $r\ell$ for some integers $r,\ell\geq 1$, then we can place $\ell$ cops on $P$ such that $P$ will be guarded in at most $r$ steps, and after these $r$ steps, only one cop is necessary to continue guarding $P$.
\end{lem}

\begin{proof}
The proof follows immediately from Lemma~\ref{lemma guard path exact}, and the fact that $\left\lceil\frac{r\ell+1}{2r+1}\right\rceil\leq \ell$ for all $r,\ell\geq 1$.
\end{proof}

Let $W=W(x)$ be the \emph{Lambert W function} or \emph{product-log function}, which is the inverse of $y=xe^x$ ($xe^x$ here will be restricted to the domain $x\geq 0$, on which $xe^x$ is injective, so $W$ is well-defined).
  We now arrive at the main result of this section, which provides a sublinear bound on the cop throttling number of a graph.

\begin{thm}\label{main}
If $G$ is a connected graph on $n$ vertices, then
	\[
	\thc(G)\leq \frac{(2+o(1))n\sqrt{W(\log n)}}{\sqrt{\log n}}.
	\]
\end{thm}

\medskip

\bpf 
Let $\tau=\sqrt{\frac{\log n}{W(\log n)}}$ and $\beta=\tau^{\tau^2}$. Let $G$ be a connected graph on $n$ vertices. First, let us consider the case where $\mathrm{diam}(G)\geq \beta\tau$.
	
We will describe how to place cops on $G$ via a recursive algorithm that decomposes $G$ into paths of length $\beta\tau$, stars, paths of length $\tau^2$, and small connected subgraphs. The paths and stars will be guarded with cops, and then we will show that there are enough free cops close to the small connected subgraphs to quickly catch the robber in any of these connected subgraphs.
	
Let $G_1=G$ and let $P_1$ be a geodesic in $G_1$ of length $\beta\tau$. Place $\beta$ cops along $P_1$ according to Lemma~\ref{lemma guard path} to guarantee $P_1$ can be guarded in $\tau$ steps. Let $G_2$ be the graph induced by $V(G_1)\setminus V(P_1)$. Now, recursively for as long as we can, let $P_i$ be a geodesic in $G_i$ of length $\beta\tau$. Place $\beta$ cops along $P_i$ according to Lemma~\ref{lemma guard path}, and let $G_{i+1}$ be the induced subgraph on $V(G_i)\setminus V(P_i)$. We can continue for, say $\ell_1$ steps, until every component in $G_{\ell_1}$ has diameter less than $\beta\tau$. Note that every vertex in $V(G_{\ell_1})$ is distance at most $\beta\tau$ from some path $P_i$.
	
We describe how to cover any large stars in $G_{\ell_1}$. Recursively, for $i\geq \ell_1$, let $v_i$ be a vertex of degree at least $\tau$ in a component of $G_i$. Place a cop at $v_i$ to guard the closed neighborhood $N_{G_i}[v_i]$, and let $G_{i+1}$ be the subgraph induced on $V(G_i)\setminus N_{G_i}[v_i]$. We can continue this until we reach some $G_{\ell_2}$ with $\Delta(G_{\ell_2})<\tau$.
	
We now will find paths of length $\tau^2$. Recursively for $i\geq \ell_2$, let $P_i$ be a geodesic in $G_i$ of length $\tau^2$. We will place $\tau$ cops on $P_i$ according to Lemma~\ref{lemma guard path} so that $P_i$ can be guarded in at most $\tau$ steps. Let $G_{i+1}$ be the induced subgraph on $V(G_i)\setminus V(P_i)$. We can continue this process until we reach some graph $G_{\ell_3}$, such that every component has diameter less than $\tau^2$. This completes the initial placement of the cops. Note that each cop covered at least $\tau$ vertices on average, so the total number of cops used is at most $\frac n \tau$. Now we will describe how to move the cops to capture the robber quickly.
	
We will guard each of the paths $P_i$ one at a time in order using the cops that were placed on the paths. It is worth noting that $P_i$ may not be a geodesic in $G$, and so $P_i$ may not initially be guardable by a single cop. Once all the vertices in $V(G)\setminus V(G_i)$ are guarded though, the robber is forced to play on $G_i$, in which $P_i$ is a geodesic, and thus, $1$-guardable.

By Lemma~\ref{lemma guard path}, each path takes at most $\tau$ steps to guard. Since each path is of length at least $\tau^2$, there are at most $\frac n {\tau^2}$ paths, so this takes at most $\frac{n}{\tau^2}\cdot \tau=\frac n \tau$ rounds. Once each path has been guarded, if the robber has not been caught yet, the robber must be in a component of $G_{\ell_3}$.
	
By the Moore bound (see e.g. \cite{MooreBound}), since $\Delta(G_{\ell_3})<\tau$ and the diameter of every component of $G_{\ell_3}$ is less than $\tau^2$, each component of $G_{\ell_3}$ has order at most $s$, where
	\bea	s&=&1+\sum_{i=1}^{\tau^2}\tau(\tau-1)^{i-1}\\
	&=& o(\tau^{\tau^2})<2\beta-2.
	\eea
Since the domination number of a component with $s$ vertices is at most $\frac s 2$, $\beta-1>\frac s 2$ cops can guard whichever component the robber ends up in. By construction, there is a path of length $\beta\tau$ with $\beta$ cops within distance $\beta\tau$ of every vertex in this component. By Lemma~\ref{lemma guard path}, only one cop need remain on each path to keep them guarded, so the $\beta-1$ other cops on this path can then guard the component containing the robber. This takes at most $2\beta\tau+1$ more steps and the robber is caught. Note that $\tau\cdot (2\beta\tau+1)=o(\tau^{2\tau^2})$ and
	\bea 
	\tau^{2\tau^2}&=&\left(\sqrt{\frac{\log n}{W(\log n)}}\right)^{\left(\frac{2\log n}{W(\log n)}\right)}\\
	&=&\left(\sqrt{\frac{W(\log n)\exp(W(\log n))}{W(\log n)}}\right)^{\left(\frac{2\log n}{W(\log n)}\right)}\\
	&=&\exp\left(\frac{1}2~W(\log n)\frac{2\log n}{W(\log n)}\right)\\
	&=&n,
	\eea
so $2\beta\tau+1=o(\frac n \tau)$. Hence, the total number of rounds it takes to capture the robber with $\frac n \tau$ cops is at most $(1+o(1))\frac n \tau$, completing the proof of this case.
	
If the diameter of $G$ is less than $\beta\tau$, then we proceed identically as in the first case, except we do not need to look for geodesics of length $\beta\tau$, and instead proceed immediately to covering large stars, and then geodesics of length $\tau^2$. We will again arrive at a graph with components that have small maximum degree and small diameter, and therefore, have order at most $s$. We then arbitrarily choose a vertex to place $\beta$ cops, and this uses at most $\frac n \tau+\beta$ cops in all. We move cops identically to the previous case, guarding all the paths in at most $\frac n \tau$ rounds, and then the $\beta>\frac s 2$ cops placed arbitrarily can then move to and guard whichever component the robber is in after at most $\beta\tau$ more steps due to the small diameter of the original graph. Since $\beta\tau+\beta=o\lp\frac{\tau^{2\tau^2}}\tau\rp=o\lp\frac n \tau\rp$, adding these together gives a bound of $(2+o(1))\frac n \tau$, finishing the proof.
\end{proof}

The following bound on the cop throttling number is slightly worse than the one in Theorem~\ref{main}, but it uses only elementary functions.

\begin{cor}
		If $G$ is a connected graph on $n$ vertices, then
	\[
	\thc(G)\leq \frac{n}{(\log n)^{1/2-o(1)}}.
	\]
\end{cor}

\begin{proof}
	We claim that $(\log x)^{\frac{\log\log\log x}{\log\log x}}=\omega\lp \sqrt{W(\log x)}\rp$. Since $\frac {\log\log\log x}{\log\log x}=o(1)$, the result will follow from Theorem~\ref{main}.

If $y=(\log x)^{\frac{\log\log\log x}{\log\log x}}$, then $\log y=\log\log\log x$. Hence, $x=\exp(e^y)$. If $z=\sqrt{W(\log x)}$, then $z^2=W(\log x)$, so $z^2e^{z^2}=\log x$, and finally $x=\exp (z^2e^{z^2})$. It is evident that $\exp(e^x)=o\lp\exp\lp x^2e^{x^2}\rp\rp$, and so we have that $(\log x)^{\frac{\log\log\log x}{\log\log x}}=\omega\lp \sqrt{W(\log x)}\rp$.
\end{proof}

Theorem \ref{main} bounds the throttling number for all graphs, but we can provide much stronger bounds for more restricted classes of graphs. We note that by using the same method as in the proof of sublinear cop numbers for connected graphs with bounded diameter in \cite{ss1004}, it is easy to see that if $G$ is a graph with diameter at most $\frac {2^{\sqrt{\log{n}}}}{\log^{3}{n}}$, then $\thc(G) = O\lp n (\log{n})^3 2^{-\sqrt{\log{n}}}\rp$. 

We finish this section with a technical lemma that bounds the throttling number for graphs that are obtained from smaller graphs by adding large stars. The result could also be used to provide improvements to a general upper bound on $\thc(G)$ in the future. 
For a connected graph $G$ of order $n$, let $S(G)$ denote the family of all  connected graphs that have the disjoint union $G\du K_{1,s}$ as a spanning subgraph for some choice of $s$, in which the copy of $G$ is induced. 

\begin{lem}\label{lemma add a large star}
	Fix $\alpha\in (0,1)$ and a positive real number $k$.  Let  $G$ be a connected graph of order $n$ with $\thc(G)\leq k n^{1-\alpha}$. If $G'\in S(G)$ is a graph of order $t$ with $t-\frac{t^\alpha}{k(1-\alpha)} >n$, then $\thc(G')\leq k t^{1-\alpha}$.
\end{lem}

\begin{proof}
	Let $f(x)=k x^{1-\alpha}$, and note that $f'(x)=k(1-\alpha) x^{-\alpha}$, which is decreasing for all $x\geq 1$.  We will show that every graph $G'$ in $S(G)$ of order $t$ has $\thc(G')\leq f(t)$. By definition, $G'$ contains a star $T$ of order at least $\frac{t^{\alpha}}{k(1-\alpha)}$ such that $G$ is the graph obtained from $G'$ by removing $T$. Since this star can be guarded by one cop,
	\begin{equation}\label{equation thc HD 1}
	\thc(G')\leq 1+\thc(G)\leq 1+f(n).
	\end{equation}
	By the mean value theorem and the fact that $f'$ is a decreasing function, we have
	\begin{equation}\label{equation thc HD 2}
	f(t)-f(n)\geq f'(t)\frac{t^{\alpha}}{k(1-\alpha)} = 1.
	\end{equation}
	Inequality \eqref{equation thc HD 2} implies that $1+f(n)\leq f(t)$, so along with inequality \eqref{equation thc HD 1}, the result holds. 
\end{proof}

 We note that a similar argument to the one above also works if we replace $n^{1-\alpha}$ with $\frac n{\log n}$ and other nice functions. Possibly the most important implication of Lemma \ref{lemma add a large star} is that it could be useful in proving a bound of the form $O(n^{1-\alpha})$ via induction. More precisely, in the inductive step, Lemma \ref{lemma add a large star} implies that one would only need to consider graphs in which every large star disconnects the graph, giving some structure to work with in a potential proof.


\subsection{Graphs with few cycles}\label{s:unicyclic}

In \cite{CRthrottle}, it was shown that 
a unicyclic graph of order $n$ has cop throttling number at most $\sqrt 6\sqrt{n}$. 
A corollary to the next result improves this unicyclic bound.
Let $f(G)$ denote the {\em vertex feedback number} of a graph $G$, i.e., the least number of vertices necessary to remove from $G$ in order to make the graph acyclic. 

\begin{prop}
A connected graph of order $n$ with vertex feedback number $f(G)$ has cop throttling number at most $2 \sqrt{n}+f(G)$.
\end{prop}

\begin{proof}
Let $G$ be a connected graph with vertex feedback number $f(G)$ and let $F$ be a set of vertices of cardinality $f(G)$ whose deletion produces an acyclic graph. Define $G_{0} = G$. We construct a sequence of graphs inductively until we reach a graph with no cycles. Given the graph $G_{i}$, pick any vertex $v_i \in F\setminus \{v_1, \dots v_{i-1}\}$ and station a single cop on $v_i$. Let $G_{i+1}$ be the graph of order $n$ obtained by deleting   edges adjacent to $v_i$  until $v_i$ is no longer in any cycle  in $G_i$. Note that this edge deletion process can be carried out so that $G_{i+1}$ is still connected.

For each $i \geq 0$, we have the inequality $\thc(G_{i}) \leq 1+\thc(G_{i+1})$. Since every cycle contains a vertex in $F$,  we have that $G_{f(G)}$ is a tree of order $n$. Since $\thc(T)\le 2\sqrt n$ for a tree $T$ of order $n$ \cite{CRthrottle}, this  gives the desired bound.
\end{proof}

We observe that if $G$ has $k$ cycles, $f(G) \leq k$ to obtain the following corollary. 

\begin{cor}
    A connected graph of order $n$ with  at most $k$ cycles has cop throttling number at most $2 \sqrt{n}+k$.  In particular, if $G$ is a connected unicyclic graph of order $n$, it follows that $G$ has cop throttling number at most $2 \sqrt{n}+1$.
\end{cor}



	\section{Cop throttling for chordal graphs}\label{s:chord}
	
	 A graph $G$ is a \emph{chordal graph} if $G$ has no induced cycle of length greater than $3$. This class of graphs is of particular interest because they are known to be cop-win \cite{AF88} and, further, since paths and trees are chordal, results obtained in this section extend to those classes as well.  
	
	We begin by establishing a result that shows that the capture time for a chordal graph is determined by the $k$-radius  (see Theorem \ref{theorem chordal capture is radius} below). 
	
	 We  need   a few  definitions and technical lemmas.
	A {\em corner} of a graph $G$ is a vertex $v$ such that there exists another vertex $u\in V(G)$, $u\neq v$, with $N[v]\subseteq N[u]$. In this case, we say {\em $u$ corners $v$} and \emph{$v$ is cornered by $u$}. A set of vertices $C$ is a {\em set of disjoint corners}  if 	every vertex in $C$ is cornered by a vertex outside of $C$. Observe that if $C$ is a set of disjoint corners, $c(G) = c(G-C)$, where $G-C=G[V(G)\setminus C]$.  Our next result  shows that removing a set of disjoint corners cannot increase the capture time of a graph and can decrease the capture time by at most one.  
	
	\begin{lem}\label{lemma disjoint corner removal}
		Let $C$ be a set of disjoint corners in a connected graph $G$ and let $S$ be a multiset of $V(G) \setminus C$. Then $\capt(G - C; S) \leq \capt(G;S) \leq \capt(G - C; S) + 1$.
	\end{lem}

	\begin{proof}
	     First, if $S$ is not a capture set of $G$, $\capt(G;S) = \capt(G-C;S) = \infty$  and the inequalities hold.
	
		Now, suppose $S$ is a capture set of $G$. We begin by showing that $\capt(G - C; S) \leq \capt(G;S)$. 
		Consider an optimal cop strategy $\psi$ on $G$ with $k$ cops starting on $S$. We adjust $\psi$ so that if any cop ever goes to a vertex $v \in C$, the cop instead goes to a vertex $u \in V(G) \setminus C$ where $u$ corners $v$ and, therefore, $N[v] \subseteq N[u]$. Call this new cop strategy $\psi'$ and observe that $\psi'$ is a legal cop strategy on $G-C$, since wherever this cop moves next is reachable from $u$. We further observe that $\psi'$ also has the property that given a fixed robber strategy, at any given time the vertices in $G-C$ occupied by the cops acting according to $\psi$ are a subset of the vertices in $G-v$ occupied by the cops acting according to $\psi'$. Then, any robber strategy that avoids $C$ is captured by $\psi'$ at least as quickly as it is by $\psi$, i.e. captured in at most $\capt(G;S)$ rounds. Since $\psi'$ never uses a vertex $v\in C$, it is also a cop strategy on $G-C$, so it follows that $\capt(G-C;S) \leq \capt(G;S)$.
		
		 Now we will show $\capt(G;S) \leq \capt(G - C; S) + 1$.
		Let $\phi$ be an optimal robber strategy on $G$, and let $\phi'$ be the robber strategy obtained by adjusting $\phi$ so that whenever the robber goes to a vertex in $v\in C$, instead they go to the vertex $u\in V(G)\setminus C$ such that $u$ corners $v$. Let us imagine for a moment that two robbers are playing simultaneously on $G$, one according to $\phi$ (the $\phi$-robber) and the other according to $\phi'$  (the $\phi'$-robber).  Note that the only time the two robbers do not occupy the same vertex is when the vertex $v$ occupied by the $\phi$-robber is in $C$.    In this case  the vertex occupied by the  $\phi'$-robber corners $v$. When the $\phi$-robber moves from a vertex $v$ in $C$ to a vertex $w$, the $\phi'$-robber is able to either move to $w$ if $w\notin C$, or move to a vertex that corners $w$ if $w\in C$. Since  the $\phi'$-robber never moves into $C$, $\phi'$ is also a strategy on  $G-C$.  Thus, there is a cop strategy with $k$ cops that captures  the $\phi'$-robber in at most $\capt(G-C;S)$ moves. If the   $\phi$-robber has not been caught yet,  they are cornered by the cop that just captured the  $\phi'$-robber, so  they can be captured in the next round. Thus $\capt(G;S) \leq \capt(G - C; S) + 1$, completing the proof.
	\end{proof}
	
	Our next lemma characterizes certain sets of vertices as sets of disjoint corners. A vertex $u$ is said to be a \emph{boundary vertex} of $v$ if $d(u,v)\geq d(w,v)$ for all $w\in N(u)$. Note that  $u$ is a boundary vertex of $v$ if and only if  no $u-v$ geodesic can be extended to a longer geodesic that ends at $v$ and includes $u$.
	
	\begin{lem}\label{lemma boundary disjoint corners}
		Fix a vertex $v$ of a connected chordal graph $G$. Then the set of boundary vertices of $v$ in $G$ is a set of disjoint corners.
	\end{lem}
	
	\begin{proof}
		Let $u$ be a boundary vertex of $v$ and let $P$ be a $u-v$ geodesic. Let $w$ be the neighbor of $u$ on $P$. We claim that $w$ corners $u$. First, notice that if $w$ is the only neighbor of $u$, $w$ corners $u$ and $d(w,v)<d(u,v)$, so $w$ is not a boundary vertex of $v$. Now, suppose $N(u) > 1$. Let $x\in N(u)\setminus \{w\}$, and consider the shortest path from $x$ to $w$ that does not use $u$. If this path is of length more than $1$, or terminates at a vertex other than $v$, this creates a chordless cycle of length more than $3$, a contradiction. Thus, $x\in N(w)$, so $w$ corners $u$. Furthermore $d(w,v)<d(u,v)$, so $w$ is not a boundary vertex of $v$. Thus, the set of boundary vertices of $v$ is a set of disjoint corners.
	\end{proof}
	


The final lemma we use for the proof of Theorem \ref{theorem chordal capture is radius} is an adaptation of a theorem in \cite{BPPR17} bounding the capture time given a covering of a graph by retracts.

\begin{lem}\label{theorem retract partition}
	Suppose that $G$ is connected and $V(G)=V_1\cup \dots \cup V_t$, where $G[V_i]$ is a retract for each $1\leq i\leq t$, and let $S$ be a multiset of $V(G)$ of order $t$. If $v_1,\dots,v_t$ are (possibly repeated) elements of $S$ such that $v_i\in V_i$ for $1\leq i\leq t$, then 
	\[
	\mathrm{capt}(G;S)\leq \max_{1\leq i\leq t} \mathrm{capt}(G[V_i];\{v_i\}).
	\]
\end{lem}

\begin{proof}
First note that if $c(G[V_i])\geq 2$ for any $i$, then $\mathrm{capt}(G[V_i];\{v_i\})=\infty$ and we are done. If each graph $G[V_i]$ is cop-win, then a single cop placed at $v_i$ can guard this graph in at most $\mathrm{capt}(G[V_i];\{v_i\})$ rounds. Thus, the strategy for the cop placed on $v_i$ is to guard $G[V_i]$. Since the $V_i$'s cover $V(G)$, after at most $\max_{1\leq i\leq t} \mathrm{capt}(G[V_i];\{v_i\})$ rounds, the entire graph $G$ is guarded, so the robber must be caught.
\end{proof}

	The preceding lemma is especially useful for chordal graphs, since all connected induced subgraphs of chordal graphs are retracts (see \cite{Q78, Q85}). We now have all the tools necessary to state our first main result on chordal graphs. The following generalizes a corollary from \cite{BPPR17} which gives the same result, but only for trees.  The {\em ball} at vertex $v$ of radius $\ell$ is $B(v,\ell)=\{w:d(v,w)\le \ell\}$.
	
	\begin{thm}\label{theorem chordal capture is radius}
		For any connected chordal graph $G$ and any set $S \subseteq V(G)$, $\capt(G;S)=\max_{v\in V(G)} d(v,S)$. 
	\end{thm}

	\begin{proof}
		Let $S \subseteq V(G)$ and let $\ell:=\max_{v\in V(G)} d(v,S)$. It is clear that $\capt(G;S) \geq \ell$ since once the cops have been placed on $S$, the robber can choose any vertex at least distance $\ell$ from every cop and stay there, avoiding capture until after $\ell$ rounds.
		
		We claim that $\capt(G;S) \leq \ell$ as well. Let $S=\{v_1,\dots,v_k\}$ and let $V_i=B(v_i,\ell)$ for each $1\leq i\leq k$. Then $V(G)=V_1\cup\dots\cup V_k$. Furthermore,  $G[V_i]$ is connected, $v_i\in V_i$ for each $1\leq i\leq k$, and since connected induced subgraphs of chordal graphs are retracts, Lemma~\ref{theorem retract partition} implies that $\capt(G;S) \leq \max_{1\leq i\leq t} \mathrm{capt}(G[V_i];\{v_i\})$.
		
		Now, fix some $1\leq i\leq k$. By Lemma \ref{lemma boundary disjoint corners}, the boundary vertices of $v_i$ in $G[V_i]$ constitute a set of disjoint corners. Since $B(v_i,\ell)\setminus B(v_i,\ell-1)$ is a subset of the set of boundary vertices of $v_i$ in $G[V_i]$ and, thus, a set of disjoint corners, Lemma \ref{lemma disjoint corner removal} implies that
		\[
		\capt(G[B(v_i,\ell)];\{v_i\})\leq \capt(G[B(v_i,\ell-1)];\{v_i\})+1,
		\]
		and iterating this, we have
		\[
		\capt(G[B(v_i,\ell)];\{v_i\})\leq\capt(G[B(v_i,1)];\{v_i\})+\ell-1=\ell.
		\]
		Thus, $\capt(G[V_i];\{v_i\})\leq \ell$ for all $1\leq i\leq k$, so $\capt(G;S)\leq \ell$, completing the proof.
	\end{proof}
	
	The next two corollaries  are immediate from Theorem \ref{theorem chordal capture is radius}.
	
\begin{cor}\label{cor chordal capture is radius}
	For any connected chordal graph $G$, $\capt_k(G)=\rad_k(G)$. 
\end{cor}


\begin{cor}
	For any connected chordal graph $G$, $\thc(G)\leq 1+\rad(G)$.
\end{cor}
	
The preceding bound is only good if the radius of $G$ is small, but the radius of chordal graphs can be as large as $\lfloor|V(G)|/2\rfloor$, as in the path $P_n$. The next corollary gives an upper bound that is much better for chordal graphs with large radius. Before we state our next corollary, we need a result due to Meir and Moon \cite{MM75}.   Let $\gamma_k(G)$ denote the \emph{$k$-distance domination number}, which is the size of the smallest set $S\subseteq V(G)$ such that $d(v,S)\leq k$ for all $v\in V(G)$.

\begin{thm}{\rm \cite{MM75}}\label{theorem k radius bound}
	For every connected graph $G$ on $n\geq k+1$ vertices, $\gamma_k(G)\leq \left\lfloor\frac{n}{k+1}\right\rfloor$.
\end{thm}

The next result now follows from Theorem \ref{theorem chordal capture is radius}.

\begin{cor}\label{tree-corollary}
	For any connected chordal graph $G$ on $n$ vertices, $\thc(G)\leq \lceil\sqrt{n}\rceil+\lfloor\sqrt{n}\rfloor-1    \le 2\sqrt n.$
\end{cor}

\begin{proof}
	By Theorem \ref{theorem k radius bound}, $\gamma_{\lceil\sqrt{n}\rceil-1}(G)\leq \lfloor\sqrt{n}\rfloor$, so $\rad_{\lfloor\sqrt{n}\rfloor}(G)\leq \lceil\sqrt{n}\rceil-1$. By Theorem \ref{theorem chordal capture is radius}, we have that
	\bea
	\thc(G)&\leq& \lfloor\sqrt{n}\rfloor+\rad_{\lfloor\sqrt{n}\rfloor}(G)\\
	&\leq &  \lceil\sqrt{n}\rceil+\lfloor\sqrt{n}\rfloor-1\\
	\null \qquad\qquad\qquad\qquad\qquad\qquad &  \leq & 2\sqrt{n}. \qquad\qquad\qquad\qquad\qquad\qquad\qquad\qquad\qedhere
	\eea
\end{proof}

It is worth noting that since trees are a subclass of chordal graphs, the previous corollary gives a generalization  of Theorem 3.9 from \cite{CRthrottle}, which states $\thc(T)\leq 2\lfloor\sqrt{n}\rfloor$ for any tree $T$ on $n$ vertices.

We end this section with a result that does not directly apply to throttling, but is nonetheless an interesting fact about the game of Cops and Robbers on chordal graphs. This result resolves an open problem from \cite{topdir}.\\

\begin{thm}
	An outerplanar graph $G$ is cop-win if and only if $G$ is connected  and chordal.
\end{thm}

 \begin{proof}
	If $G$ is connected and chordal, then $G$ is cop-win \cite{AF88}.  If $G$ is not connected, then $G$ is not cop-win.  So assume that $G$ is outerplanar, connected, and not chordal.  Thus, there  exists a subset $U \subseteq V(G)$ with $|U| > 3$ such that $G[{U}]$ is a cycle.
	Then the robber can start on a vertex of $U$ that is not adjacent to the cop's starting vertex, since all vertices on the cycle are adjacent to only two other vertices on the cycle, and all vertices outside the cycle are only adjacent to at most two vertices on the cycle because $G$ is outerplanar.
	If the cop ever moves to a vertex that is adjacent to the robber's current position, the robber can always move to a vertex that is non-adjacent to the cop's current position,  since the robber's current vertex has two neighbors on the cycle and the cop is only ever adjacent to at most two vertices on the cycle (including the robber's current vertex). Otherwise, the robber waits at their current vertex for the cop to reach one of their neighbors.
\end{proof}


\section{Product throttling for Cops and Robbers}\label{sprodthrot}

In this section we consider product cop throttling, to better represent the idea of optimizing resources in a situation where the most relevant  metric is person-hours or some similar measure. 
This extends the notion of speed-up explored in \cite{karp-ramachandran,luccio-pagli-conference,luccio-pagli,luccio-pagli-decontamination}.
We begin by comparing the product cop throttling number $\thcx(G)$ and the cop throttling number $\thc(G)$.

\begin{rem}\label{sum=prod}
Let $G$ be a graph. For any capture set $S$,  
\[\thcx(G;S)=|S|(1+ \capt(G;S))=|S|+|S|\capt(G;S)\ge |S|+\capt(G;S)=\thc(G;S),\] so $  \thcx(G)\ge \thc(G)$.  Furthermore,  $\thcx(G)=\thc(G)$ if and only if $\thc(G)=\thc(G,1)$ or $\thc(G)=\thc(G,n)$; i.e., the cop throttling number can be realized with a single cop or a cop on every vertex. 
 \end{rem}

Let $G$ be a graph of order $n$.  If $n\ge 2$, then $\thcx(G)\ge 2$.  Clearly $\thcx(G)\le n$ since we can place a cop on each vertex.  By using the minimum number of cops needed to capture the robber, $\thcx(G)\le c(G)(1+\capt(G))$. Since a dominating set of cardinality less than $n$ has capture time equal to one, $\thcx(G)\le 2\gamma(G)$, where $\gamma(G)$ denotes the domination number of $G$.   If $G$ has no isolated vertices, then $\thcx(G;S)\le n$ can be achieved with nonzero capture time because $\gamma(G)\le \lf\frac n 2\rf$. 

\begin{prop}\label{thcthcxbounds}
For any graph $G$, $\thc(G) \leq \thcx(G) \leq \left \lfloor \frac{(\thc(G) + 1)^2}{4} \right \rfloor$.
\end{prop}

\begin{proof}
The first inequality is justified in Remark \ref{sum=prod}.  

Let $k$ and $\ell$ be integers such that $\thc(G,k)=\thc(G)$ and $\thcx(G,\ell)=\thcx(G)$. Then by the definitions of sum and product throttling, and the arithmetic mean-geometric mean (AM-GM) inequality,
\[
\thcx(G)=\ell(1+\capt_\ell(G))\leq k(1+\capt_k(G))\leq \left(\frac{k+(1+\capt_k(G))}2\right)^2=  \frac{(1+\thc(G))^2}4. \qedhere
\]
\end{proof}
 
\begin{cor}\label{thcthcxLow} Let $G$ be a graph.
\ben[$(1)$]
	\item\label{thcx1} $\thcx(G)=1$ if and only if  $ \thc(G)=1$ if and only if $G=K_1$.  
	\item\label{thcx2} $\thcx(G)=2$ if and only if $ \thc(G)=2$ if and only if either $G=2K_1$ or $\gamma(G)=1$.  
	\item\label{thcx3} $\thcx(G)=3$ if and only if   
	 	$G$ satisfies one of the following conditions:
		\begin{enumerate}[(a)]
			\item \label{30} $G=3K_1$ or $G=K_1\du K_2$.
			\item \label{12} $\gamma(G)\ge 3$ and there exists $z\in V(G)$ such that
 				\ben[(i)]
 					\item \label{12a} for all $v\in V(G)$, $\dist(z,v)\le 2$, and 
					 \item \label{12b} for all $w\in V(G)\setminus N[z]$, there is a vertex $u\in N[z]$ such that $N[w]\subset N[u]$.  This condition says that for $w\in V(G)\setminus N[z]$ there is a vertex $u\in N(z)$ such that $w$ is cornered by $u$ (see Section \ref{s:chord}). 

				\een
		\end{enumerate}
	\item\label{thcx4} $\thcx(G)=4$ if and only if $G$ satisfies one of the following conditions:
		\ben[(a)]
			\item $|V(G)|=4$ and $\gamma(G)\ge  2$.
			\item $\gamma(G)= 2$ and $|V(G)|\ge 4$.
			\item $c(G)=1$ and $\capt(G)=3$.
		\een
\een
\end{cor}
\bpf  \eqref{thcx1} and \eqref{thcx2}:  For $r=1,2$, $\thc(G)=r$ if and only if  $\thc(G) = r$ follows from Proposition~\ref{thcthcxbounds}.  Graphs with $\thc(G) \in \{1, 2\}$ were characterized in \cite{CRthrottle}.

\eqref{thcx3}:  There are exactly two ways $\thcx(G;S)=3$ can be achieved: $|V(G)|=|S|=3$ or both $c(G)=1$ and $\capt(G)=2$.   Requiring  $|V(G)|=|S|=3$ and $\thcx(G)>2$ is equivalent to $G=3K_1$ or $G=K_1\du K_2$. 
 It is shown in the proof of Theorem 4.1 in \cite{CRthrottle} that the graphs for which $c(G)=1$ and $\capt(G)=2$ are those in \eqref{thcx3}(b). 
 
\eqref{thcx4}:  There are exactly three ways $\thcx(G;S)=4$ can be achieved:  $|V(G)|=|S|=4$, $\gamma(G)=2$, or both $c(G)=1$ and $\capt(G)=3$. To ensure $\thcx(G)\ge 4$, if $|V(G)|=4$ we need the condition $\gamma(G)\ge  2$, and if $\gamma(G)= 2$ we need the condition $|V(G)|\ge 4$.   \epf

We now turn our attention to determining the product cop throttling number for chordal graphs. We find the following characterization of connected chordal graphs useful   \cite[Proposition 5.5.1]{GTbook}.  
	A connected chordal graph $G$  can be built successively by adding cliques with vertex sets  $X_1, \dots, X_k$ in such a way that $X_i\cap (\cup_{j=1}^{i-1} X_j)\ne \emptyset$ and there exists an $\ell$ with $1\le \ell\le i-1$ such  that $X_i\cap(\cup_{j=1}^{i-1} X_j)\subseteq X_\ell$. This implies $G[\cup_{j=1}^i X_j]$ is a connected chordal graph, $X_i\cap(\cup_{j=1}^{i-1} X_j)$ induces a clique,  and $V(G)=\cup_{j=1}^k X_j$. We will call the ordered sets $X_1,\dots, X_k$ a \emph{clique decomposition} of $G$.

\begin{lem}\label{lemma geodesic in chordal graph}
	Let $P$ be a geodesic in a connected chordal graph $G$. Then $P$ can be obtained from $G$ by repeated corner deletions.
\end{lem}

\begin{proof}
	Let $e_1,\dots,e_d$ be the edges of $P$ (indexed in order), and let $X_1,\dots,X_k$ be a clique decomposition of $G$. We can assume without loss of generality that $e_i\in X_i$ for $1\leq i\leq d$ since $P$ is a geodesic. Now, let $z\in X_k\setminus \left(\bigcup_{i=1}^{k-1}X_i\right)$ and $w\in X_k\cap \left(\bigcup_{i=1}^{k-1}X_i\right)$. Note that $N[z]\subseteq N[w]$, so $z$ is a corner. This implies that the graph $H=G[X]$ with $X=\bigcup_{i=1}^dX_i$  (i.e., the cliques that contain the path $P$) can be obtained from $G$ by repeated corner deletions. 
	
  Now  suppose there exists a vertex $z\in X_i\setminus (X_{i-1}\cup X_{i+1})$ for some $i$ with $1\leq i\leq d$ (assume for simplicity $X_0=X_{d+1}=\emptyset$).  Then  $N[z]=X_i\subseteq N[w]$ for $w\in X_i\cap (X_{i-1}\cup X_{i+1})$. This implies $z$ is a corner, and this property is preserved when deleting vertices from the set $X_i\setminus (X_{i-1}\cup X_{i+1})$. Similarly, if $|X_i\cap X_{i+1}|\geq 2$ for some $1\leq i<d$, then for any $w,z\in X_i\cap X_{i+1}$, we have $N[w]=N[z]$, and so they are both corners. Thus, $P$ can be obtained from $H$ via repeated corner deletions, completing the proof.
\end{proof}

We can now determine exactly the value of $\thcx(G)$ for connected chordal graphs.

\begin{thm}\label{theorem multiplicative chordal bound}
	For any connected chordal graph $G$, $\thcx(G)=1+\rad(G)$.
\end{thm}

\begin{proof}
	Applying Corollary \ref{cor chordal capture is radius} with $k=1$ yields $\thcx(G)\le 1+\rad(G)$.	
	For the reverse inequality, let $d=\diam(G)$, and $P=P_{d+1}$ be a diametric path in $G$. Since diametric paths are necessarily geodesics, we know via Lemma \ref{lemma geodesic in chordal graph} that $P$ can be obtained from $G$ by repeated corner deletions. Then, applying Lemma \ref{lemma disjoint corner removal} gives us that $\capt_k(P)\leq \capt_k(G)$ for any $k$, and so $\thcx(P)\leq \thcx(G)$. By Corollary \ref{cor chordal capture is radius}, $\capt_k(P)=\rad_k(P)\geq\frac{d+1-k}{2k}$, where the inequality comes from the fact that $k(2\rad_k(P)+1)\geq d+1$. Then for all $S\subseteq V(P)$ with $|S|\geq 2$,
	\bea
	\thcx(P;S)&\geq &|S|\left(1+\frac{d+1-|S|}{2|S|}\right)\\
	&=&|S|+\frac{d+1-|S|}{2}\\
	&=&\frac{|S|}2+\frac{d+1}2\\
	&\geq & 1+\rad(G).
	\eea
	By Corollary \ref{cor chordal capture is radius},  $\thcx(P;S)\ge 1+\rad(P)$ also holds for any $S\subseteq V(P)$ of size $1$.
\end{proof}

In examples for which  the cop throttling number has been determined, the minimum often occurs when the number of cops and the capture time are approximately equal. In contrast, for graphs $G$ for which    $\thcx(G)$ has been determined, the minimum is often achieved  when the number of cops is as small as possible, i.e., $c(G)$, and the capture time may be larger. 
For example, it follows from Theorem  \ref{theorem multiplicative chordal bound} that the product cop throttling number for a path on $n$ vertices is achieved with one cop  while the capture time is $\lfloor\frac{n}{2}\rfloor$. This is in sharp contrast to   cop throttling for a path, where approximately $\sqrt{\frac n 2}$ cops are used to realize the cop throttling number and the capture time is also approximately $\sqrt{\frac n 2}$  \cite{CRthrottle}.
Further, it can also be the case that in realizing $\thcx(G)$  it is best to have a small capture time and  a larger number of cops, i.e., capture time equal to one with $\gamma(G)$ cops.  An example of this is provided by a graph in the family $H(n)$ defined in \cite{BGHK09}, where it is shown that $\capt_1(H(n))=n-4$.  For $H(11)$, shown in Figure \ref{fig:H11}, $\capt_1(H(11))=7$, but vertices 5 and 7 dominate the graph, so  $\thcx(H(11))=4$ and $\thc(H(11))=3$.  
However, this is not always the case and the next example provides a family of graphs $G$ for which both $\thcx(G,c(G))>\thcx(G)$ and $\thcx(G,\gamma(G))>\thcx(G)$ for sufficiently large order.

\begin{figure}[h]
\centering
\scalebox{.4}{\includegraphics{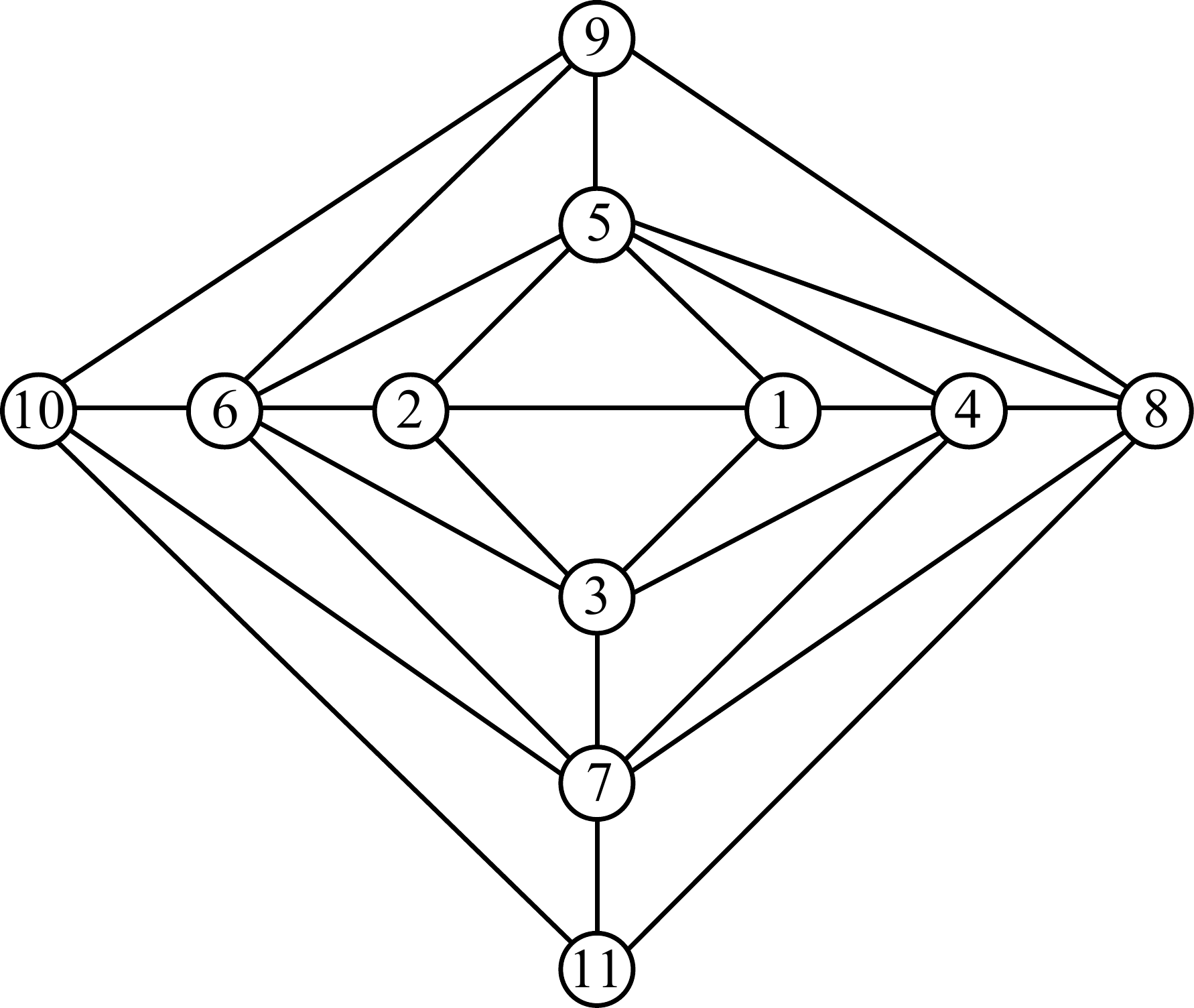}}
\caption{The graph $H(11)$.  \label{fig:H11}}
\end{figure}

\begin{ex} Fix a positive integer $\ell$, and let $M'(\ell)$ be the graph obtained from the disjoint union of $C_4$ with three copies of $P_\ell$ by pairing the three paths with three distinct vertices of $C_4$ and adding an edge from an endpoint of each path to the paired vertex of $C_4$.  Let  $M(\ell)$ be the result of appending a leaf to every vertex of $M'(\ell)$.   The graph $M(3)$ is shown in Figure \ref{fig:L3}.  Note that the order of $M(\ell)$ is $6\ell+8$.
\end{ex} 

\begin{figure}[h]
\centering
\scalebox{.45}{\includegraphics{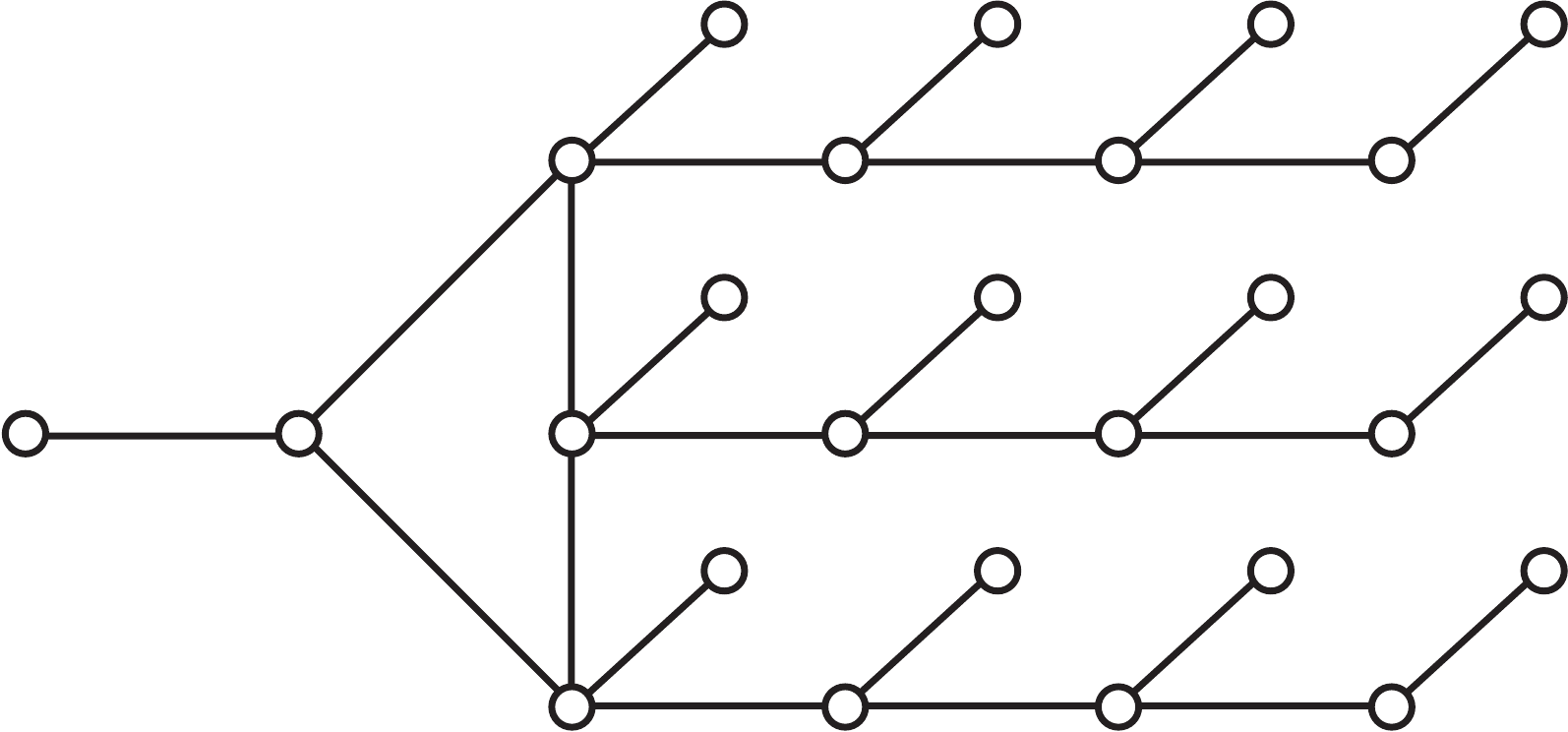}}
\caption{The graph $M(3)$. \vspace{-8pt} \label{fig:L3}}
\end{figure}

\begin{thm}
 There exist infinitely many graphs $G$ for which $\thcx(G)$ is not achieved by any set of size $\gamma(G)$ nor by any set of size $c(G)$.   In particular, this is the case for $G=M(\ell)$ with $\ell\ge 7$.
\end{thm}

\begin{proof}
The cop number of $M(\ell)$ is two, since it is unicyclic and it does not have a universal vertex.  Suppose first that we use two cops.  Consider a part consisting of a vertex of $C_4$, its attached path, and the adjacent leaves.  Observe that one of these three parts must start without a cop on it.  The distance between any vertex not in this part and the leaf attached to the path endpoint at distance $\ell$ from the $C_4$ is at  least $\ell+2$,  so  
$\capt_2(M(\ell))\geq \ell+2$ and $\thcx(M(\ell),c(M(\ell)))\ge2(\ell+3)$.

It is immediate that  $\gamma(M(\ell))=3\ell+4$ because each leaf or its neighbor must be in a dominating set, and the capture time for a dominating set that does not include all vertices is one, so $\thcx(M(\ell),\gamma(M(\ell)))=(3\ell+4)(1+1)> 2(\ell+3)$.

Now, let $S$ consist of the three vertices on the three paths each at distance $\lc\frac{\ell+3}2\rc$ from the vertex at the end of the original $P_\ell$. 
Then $\capt(M(\ell);S)= \lc\frac{\ell+3}2\rc$ since every vertex is within distance $\lc\frac{\ell+3}2\rc$ of a cop, and the cops can clear the entire graph in this many rounds.  Thus,  $\thcx(M(\ell),3)\leq 3\lp1+\lc\frac{\ell+3}2\rc\rp<2(\ell+3)$ for $\ell\ge 7$.
\end{proof}

Finally, we show that if the upper bound in Proposition \ref{thcthcxbounds} is tight for a graph $G$, then there are specific restrictions on the number of cops than can be used in $G$ to realize $\thc(G)$. 
To facilitate a discussion of these restrictions, we define some terms. An ordered pair $(k, p)$ is a \emph{throttling point of $G$} if there exists a capture set $S \subseteq V(G)$ such that $|S| = k$ and $\capt(G; S) = p$. Suppose $(k,p)$ is a throttling point of $G$. Then $(k,p)$ is \emph{sum-minimum} if $k + p = \thc(G)$ and $(k,p)$ is \emph{product-minimum} if $k(1+p) = \thcx(G)$.  If $(k,p)$ is sum-minimum (respectively, product-minimum), then $\thc(G)=\thc(G,k)$ (respectively, $\thcx(G)=\thcx(G,k)$).

\begin{prop}
Suppose $G$ is a graph with $\thc(G) = q$ and let $I(q)$ be a set of ordered pairs, defined as
\begin{eqnarray*}
I(q) =
\begin{cases}
    \left\{\left(\frac{q+1}{2}, \frac{q-1}{2}\right)\right\}, & \text{if } q \text{ is odd;}\\
    \left\{\left(\frac{q}{2}, \frac{q}{2}\right), \left(\frac{q+2}{2}, \frac{q-2}{2}\right)\right\}, & \text{if } q \text{ is even.}
\end{cases}
\end{eqnarray*}
Then $\thcx(G) = \left \lfloor \frac{(q+1)^2}{4} \right \rfloor$ if and only if every sum-minimum throttling point of $G$ is contained in $I(q)$ and one such throttling point is also product-minimum.
\end{prop}

\begin{proof}
We have
\begin{eqnarray}
\nonumber
\thcx(G) &=& \min\{x(1+y) : (x,y) \text{ is a throttling point of }G\}\\ 
&\leq &\min\{x(1 + q-x) : (x, q-x) \text{ is a sum-minimum throttling point of }G\}\label{ineq1}\\ 
&\leq &\max\{x(1 + q-x) : (x, q-x) \text{ is a sum-minimum throttling point of }G\}\label{ineq2}\\ 
&\leq &\max\{x(1+y) : x,y \in \mathbb{N} \text{ and } x + y = q\} = \left \lfloor \frac{(q+1)^2}{4} \right \rfloor.  \label{ineq3}
\end{eqnarray}
Thus, $\thcx(G) = \left \lfloor \frac{(q+1)^2}{4} \right \rfloor$ if and only if every inequality above is an equality. It is clear that equality holds in \eqref{ineq1} if and only if there exists a sum-minimum throttling point of $G$ that is also product-minimum. Note that the maximum value of a finite set is equal to the minimum value of the set if and only if the set has exactly one element. Therefore, equality holds in \eqref{ineq2} if and only if $a(1+q-a) = b(1+q-b)$ for all sum-minimum throttling points $(a, q-a)$ and $(b, q-b)$ of $G$. If $a \neq b$, then $a(1+q-a) = b(1+q-b)$ if and only if $b = 1+q-a$. 
So \eqref{ineq2} is an equality if and only if there are only two possible sum-minimum throttling points (namely, $(a, q-a)$ and $(1+q-a, a-1)$ for some fixed $a \in \{1, 2, \ldots, q\}$). Finally, equality holds in \eqref{ineq3} if and only if an ordered pair $(x,y)$ that realizes $\max\{x(1+y) \ | \ x,y \in \mathbb{N} \text{ and } x + y = q\}$ is a sum-minimum throttling point of $G$. It is easy to see that the ordered pairs that realize $\max\{x(1+y) \ | \ x,y \in \mathbb{N} \text{ and } x + y = q\}$ are exactly the points in $I(q)$. Therefore, equality simultaneously holds in \eqref{ineq1}, \eqref{ineq2}, and \eqref{ineq3}  if and only if every sum-minimum throttling point of $G$ is a point in $I(q)$ and one of these throttling points is also product-minimum.
\end{proof}

 The definition of $I(q)$ in the previous proposition essentially
 characterizes when equality holds in an integer two-item version of  the AM-GM inequality.





\end{document}